\documentclass[a4paper,10pt]{amsart}

\usepackage[utf8]{inputenc}

\usepackage[foot]{amsaddr}
\usepackage{amsmath}
\usepackage{amstext}
\usepackage{amsfonts}
\usepackage{amssymb}
\usepackage{enumitem}

\usepackage{graphicx}
\usepackage{xcolor}
\usepackage{dsfont}
\usepackage{nicefrac}

\usepackage{colonequals}

\usepackage{hyperref}
\hypersetup{colorlinks=true, linkcolor=blue}

\newcommand{\inner}[3][]{( #2 , #3 )_{#1}}
\newcommand{\innerb}[3][]{\big(#2,#3\big)_{#1}}
\newcommand{\innerB}[3][]{\Big(#2,#3\Big)_{#1}}
\newcommand{\dual}[3][]{\langle #2 , #3\rangle_{#1}}
\newcommand{\diff}[1]{\,\mathrm{d}#1}
\newcommand{\R}{\mathbb{R}}
\newcommand{\N}{\mathbb{N}}
\newcommand{\F}{\mathcal{F}}
\renewcommand{\P}{\mathbf{P}}
\newcommand{\E}{\mathbf{E}}

\renewcommand{\L}{\mathcal{L}}
\renewcommand{\exp}[1]{\mathrm{e}^{#1}}
\newcommand{\exptext}[1]{\mathrm{exp}(#1)}

\newcommand{\exptextB}[1]{\mathrm{exp}\Big(#1\Big)}
\newcommand{\exptextbigg}[1]{\mathrm{exp}\bigg(#1\bigg)}
\newcommand{\domain}[1]{\mathcal{D}\left( #1 \right)}

\newcommand{\Lr}[2]{L_{#1}(#2)}

\DeclareMathOperator*{\argmin}{arg\,min}

\numberwithin{equation}{section}
\theoremstyle{plain}
\newtheorem{theorem}{Theorem}[section]
\newtheorem{lemma}{Lemma}[section]
\newtheorem{assumption}{Assumption}

\theoremstyle{definition}
\newtheorem{remark}{Remark}[section]

\begin{document}

\title[Sub-linear convergence of SPI in Hilbert space]{Sub-linear convergence of a stochastic proximal iteration method in Hilbert space}

\author[M.~Eisenmann]{Monika Eisenmann}
\email{monika.eisenmann@math.lth.se}

\author[T.~Stillfjord]{Tony Stillfjord}
\email{tony.stillfjord@math.lth.se}

\author[M. Williamson]{M\aa{}ns Williamson}
\email{mans.williamson@math.lth.se}

\address{  Centre for Mathematical Sciences\\
  Lund University\\
  P.O.\ Box 118\\
  221 00 Lund, Sweden}

\thanks{
This work was partially supported by the Wallenberg AI, Autono\-mous Systems 
and Software Program (WASP) funded by the Knut and Alice Wallenberg Foundation. 
The computations were enabled by resources provided by the Swedish National 
Infrastructure for Computing (SNIC) at LUNARC partially funded by the Swedish 
Research Council through grant agreement no. 2018–05973.
The authors would like to thank Eskil Hansen for valuable feedback.}


\begin{abstract}
We consider a stochastic version of the proximal point algorithm for optimization 
problems posed on a Hilbert space. A typical application of this is supervised 
learning. While the method is not new, it has not been extensively analyzed in this 
form. Indeed, most related results are confined to the finite-dimensional 
setting, where error bounds could depend on the dimension of the space. On the 
other hand, the few existing results in the infinite-dimensional setting only prove 
very weak types of convergence, owing to weak assumptions on the problem. In 
particular, there are no results that show convergence with a rate. In this article, 
we bridge these two worlds by assuming more regularity of the optimization 
problem, which allows us to prove convergence with an (optimal) sub-linear rate also in an 
infinite-dimensional setting.
In particular, we assume that the objective function is the expected value of a family
of convex differentiable functions. While we require that the full objective function
is strongly convex, we do not assume that its constituent parts are so. Further, we
require that the gradient satisfies a weak local Lipschitz continuity property, where
the Lipschitz constant may grow polynomially given certain guarantees on the variance
and higher moments near the minimum.
We illustrate these results by discretizing a concrete 
infinite-dimensional classification problem with varying degrees of accuracy.
\end{abstract}

\maketitle

\section{Introduction}
We consider convex optimization problems of the form
\begin{equation} \label{eq:main_problem}
  w^* = \argmin_{w \in H} F(w),
\end{equation}
where
\begin{equation*}
  F(w) = \E_{\xi} [ f(w, \xi) ].
\end{equation*}
The main applications we have in mind are supervised learning tasks. In such a 
problem, a set of data samples $\{x_j\}_{j=1}^{n}$ with corresponding labels 
$\{y_j\}_{j=1}^{n}$ is given, as well as a classifier $h$ depending on the 
parameters $w$. The goal is to find $w$ such that $h(w,x_j) \approx y_j$ for 
all $j \in \{1,\dots,n\}$. This is done by minimizing
\begin{equation} \label{eq:ML_functional}
  F(w) = \frac{1}{n} \sum_{j = 1}^{n}{\ell(h(w, x_j), y_j)},
\end{equation}
where $\ell$ is a given loss function. We refer to, e.g.,  
Bottou, Curtis \& Nocedal \cite{BottouCurtisNocedal.2018} for an overview. In 
order to reduce the 
computational costs, it has been proved to be useful to split
$F$ into a collection of functions $f$ of the type
\begin{align*}
    f(w,\xi) = \frac{1}{|B_{\xi}|} \sum_{j \in B_{\xi} } {\ell(h(w, x_j), y_j)},
\end{align*}
where $B_{\xi}$ is a random subset of $\{1,\dots,n\}$, referred to as 
a batch. In particular, the case of $|B_{\xi}| = 1$  is interesting for applications,
as it corresponds to a separation of the data into single samples.

A commonly used method for such problems is the stochastic gradient method (SGD), given by the iteration
\begin{equation*}
  w^{k+1} = w^k - \alpha_k \nabla f(w^k, \xi^k),
\end{equation*}
where $\alpha_k >0$ denotes a step size, $\{\xi^k\}_{k \in \N}$ is a family of jointly 
independent random variables and $\nabla$ denotes
the G\^{a}teaux derivative with respect to the first variable. The idea is that in each 
step we choose a random part $f(\cdot, \xi)$ of $F$ and go in the direction of the 
negative gradient of this function. 
SGD corresponds to a stochastic version of the explicit (forward) Euler scheme 
applied to the gradient flow
\begin{equation*}
   \dot{w} = - \nabla F(w). 
\end{equation*}
This differential equation is frequently stiff, which means that the method often suffers 
from stability issues.

The restatement of the problem as a gradient flow suggests that we could avoid 
such stability problems by instead considering a stochastic version of 
implicit (backward) Euler, given by
\begin{equation*}
  w^{k+1} = w^k - \alpha_k \nabla f(w^{k+1}, \xi^k).
\end{equation*}
In the deterministic setting, this method has a long history under the name 
\emph{proximal point method}, because it is equivalent to 
\begin{equation*}
  w^{k+1} = \argmin_{w \in H} \Big\{ \alpha F(w) + \frac{1}{2} \| w - w^k\|^2
  \Big\} = \text{prox}_{\alpha F}(w^k),
\end{equation*}
where 
\begin{equation*}
  \text{prox}_{\alpha F}(w^k) = (I + \alpha \nabla F)^{-1} w^k.
\end{equation*}
The proximal point method has been studied extensively 
in the infinite dimensional but deterministic case, beginning with the work of 
Rockafellar~\cite{Rockafellar.1976}. Several convergence results and connections
to other methods such as the Douglas--Rachford splitting are collected in 
Eckstein \& Bertsekas~\cite{EcksteinBertsekas.1992}.

Following Ryu \& Boyd~\cite{RyuBoyd.2016}, we will refer to the stochastic version 
as \emph{stochastic proximal iteration} (SPI). We note that the computational cost 
of 
one SPI step is in general much higher than for SGD, and indeed often infeasible. 
However, in many special cases a clever reformulation can result in very similar 
costs. If so, then SPI should be preferred over SGD, as it will converge more 
reliably.
We provide such an example in Section~\ref{section:numerical_experiments}.

The main goal of this paper is to prove sub-linear convergence of the type 
\begin{equation*}
  \E\big[ \|w^k - w^* \|^2 \big] \le \frac{C}{k}
\end{equation*}
in an infinite-dimensional setting, i.e.\ where $\{w^k\}_{k \in \N}$ and $w^*$ are elements 
in a Hilbert space $H$. As shown in e.g.~\cite{AgarwalEtal.2012,RaginskyRakhlin.2011},
this is optimal in the sense that we cannot expect a better asymptotic rate even in
the finite-dimensional case.

Most previous convergence results in this setting only provide guarantees for 
convergence, without an explicit error bound. The convergence is usually also in a rather 
weak norm. This is mainly due to weak assumptions on the involved functions 
and operators. 
Overall, little work has been done to consider SPI in an infinite dimensional space. 
A few exceptions are given by Bianchi \cite{Bianchi.2016}, where maximal
monotone operators $\nabla F \colon H \to 2^H$ are considered and weak 
ergodic convergence and norm convergence is proved. In
Rosasco et al. \cite{RosascoVillaVu.2020}, the authors work with an infinite
dimensional setting and an implicit-explicit splitting where $\nabla F$ is 
decomposed in
a regular and an irregular part. The regular part is considered explicitly but with a
stochastic approximation while the irregular part is used in a deterministic 
proximal step. They prove both $\nabla F(w^k) \to \nabla F(w^*)$ and $w^k \to 
w^*$ in $H$ as $k \to \infty$. Without further assumptions, neither of these
approaches yield convergence rates.

In the finite-dimensional case, stronger assumptions are typically 
made, with better convergence guarantees as a result. Nevertheless, for the SPI 
scheme in particular, we are only aware of the unpublished 
manuscript~\cite{RyuBoyd.2016}, which suggests $\nicefrac{1}{k}$ convergence in 
$\R^d$. 
Based on~\cite{RyuBoyd.2016}, the implicit method has also been considered in a few other works:
In Patrascu \& Necoara 
\cite{PatrascuNecoara.2017}, a SPI method with additional constraints on the 
domain was studied. A slightly
more general setting that includes the SPI has been considered in Davis \&
Drusvyatskiy \cite{DavisDrusvyatskiy.2019}. Toulis \& Airoldi and Toulis et al.
studied such an implicit scheme in 
\cite{ToulisAiroldi.2015,ToulisAiroldi.2016,ToulisRennieAiroldi.2014}.

Whenever using an implicit scheme, it is essential to solve the appearing implicit
equation effectively. This can be impeded by large batches for the stochastic
approximation of $F$. On the other hand, a larger batch improves the accuracy of
the approximation of the function. In Toulis, Tran \& Airoldi
\cite{ToulisTranAiroldi.2016,TranToulisAiroldi.2015} and Ryu \& 
Yin \cite{RyuYin.2019}, a compromise was found by
solving several implicit problems on small batches and taking the average of
these results. This corresponds to a sum splitting. Furthermore, implicit-explicit
splittings can be found in Patrascu \& Irofti
\cite{PatrascuIrofti.2020}, Ryu \& Yin \cite{RyuYin.2019}, Salim et al. 
\cite{SalimBianchiHachem.2019}, Bianchi \& Hachem \cite{BianchiHachem.2016} 
and Bertsekas \cite{Bertsekas.2011}.
A few more related schemes have been considered in Asi \&
Duchi \cite{AsiDuchi.2019B,AsiDuchi.2019} and Toulis, Horel \& Airoldi
\cite{ToulisHorelAiroldi.2020}.
More information about the complexity of solving these kinds of implicit 
equations and the
corresponding implementation can be found in Fagan \& Iyengar \cite{FaganIyengar.2020} and
Tran, Toulis \& Airoldi in~\cite{TranToulisAiroldi.2015}.

Our aim is to bridge the gap between the ``strong finite-dimensio\-nal'' 
and ``weak infinite-dimensional'' settings, by extending the approach 
of~\cite{RyuBoyd.2016} to the infinite-dimensional case. We also further extend the results
by allowing for more general Lipschitz conditions on $\nabla f(\cdot,\xi)$, provided that
sufficient guarantees can be made on the integrability near the minimum $w^*$.
These strong convergence results can then be applied to, e.g., the setting where there is an 
original infinite-dimensional optimization problem which is subsequently 
discretized into a series of finite-dimensional problems. Given a reasonable 
discretization, each of those problems will then satisfy the same convergence 
guarantees.
We will follow~\cite{RyuBoyd.2016} closely, because their approach is sound. 
However, several arguments no longer work in the infinite-dimensional case
(such as the unit ball being compact, or a linear operator having a minimal 
eigenvalue) and we fix those. Additionally, we simplify several of the remaining 
arguments, provide many omitted, but critical, details and extend the results to 
less bounded operators. 

A brief outline of the paper is as follows. The main assumptions that we make are 
stated in Section~\ref{section:assumptions}, as well as the main theorem. Then we 
prove a number of preliminary results in Section~\ref{section:preliminaries}, before 
we can tackle the main proof in Section~\ref{section:main_theorem}. In 
Section~\ref{section:numerical_experiments} we describe a numerical experiment 
that illustrates our results, and then we summarize our findings in 
Section~\ref{section:conclusions}.

\section{Assumptions and main theorem}\label{section:assumptions}
Let $(\Omega, \F, \P )$ be a complete probability space and let $\{\xi^k\}_{k \in 
\N}$ be a family of jointly independent random variables on $\Omega$. Each realization of 
$\xi^k$ corresponds to a different batch. 
Let $(H, \inner{\cdot}{\cdot}, \|\cdot \|)$ be a real Hilbert space and $(H^*, 
\inner[H^*]{\cdot}{\cdot}, \| \cdot \|_{H^*} )$ its dual. 
Since $H$ is a Hilbert space, there exists an isometric isomorphism $\iota \colon 
H^* \to H$ such that $\iota^{-1} \colon H \to H^*$ with $\iota^{-1}: u \mapsto 
\inner{u}{\cdot}$. Furthermore, the dual pairing is denoted by $\dual{u'}{u} = u'(u)$ 
for $u' \in H^*$ and $u \in H$. It satisfies
\begin{align*}
  \dual{\iota^{-1} u }{v} = \inner{u}{v}
  \quad
  \text{and}
  \quad 
  \dual{u' }{v} = \inner{\iota u'}{v}, 
  \quad u,v \in H, u' \in H^*.
\end{align*}
We denote the space of linear bounded operators mapping $H$ into $H$ by 
$\L(H)$. For a symmetric operator $S$, we say that it is positive if $\inner{Su}{u} \geq 
0$ for all $u \in H$. It is called strictly positive if $\inner{Su}{u} > 0$ for all $u \in 
H$ such that $u \neq 0$.

For the function $f(\cdot, \xi) \colon H \times \Omega \to (-\infty, \infty]$, we use
$\nabla$, as in $\nabla f(u, \xi)$, to denote differentiation with respect to the 
first variable. When we present an argument that holds almost surely, 
we will frequently omit $\xi$ from the notation and simply write $f(u)$ rather than 
$f(u, \xi)$. Given a random variable $X$ on $\Omega$, we 
denote the expectation with respect to $\P$ by $\E[X]$. We use sub-indices, such 
as in $\E_{\xi}[\cdot]$, to denote expectations with respect to the probability 
distribution of a certain random variable.

For the family of jointly independent random variables $\{\xi^k\}_{k \in \N}$, we are 
interested in the total expectation
\begin{equation*}
  \E_{k}\big[ \|X\|^2 \big] = \E_{\xi^1}\big[ \E_{\xi^2}\big[ \cdots 
  \E_{\xi^{k}} \big[ \|X \|^2 \big] \cdots \big]\big].
\end{equation*}
Since the random variables $\{\xi^k\}_{k \in \N}$ are jointly independent, and 
$w^k$ only depends on $\xi^j$, $j \le k-1$, this expectation coincides with the 
expectation with respect to the joint probability distribution of $\xi^1, \ldots, 
\xi^{k-1}$.
In the rest of the paper, it often occurs that a statement does not involve an expectation 
but contains a random variable. Where it does not cause any confusion, such a statement is 
assumed to hold almost surely even if this is not explicitly stated. 

We consider the stochastic proximal iteration (SPI) scheme given 
by
\begin{align}\label{eq:scheme}
  w^{k+1} = w^k - \alpha_k \iota \nabla f(w^{k+1}, \xi^k) \quad \text{ in } H, 
  \quad
  \quad w^1 = w_1 \quad \text{ in } H,
\end{align}
for minimizing
\begin{equation*}
  F(w) = \E_{\xi} [ f(w, \xi) ], 
\end{equation*}
where $f$ and $F$ fulfill the following assumption.

\begin{assumption}\label{ass:f}
  For a random variable $\xi$ on $\Omega$, let the function $f(\cdot, \xi) \colon \Omega 
  \times H \to (-\infty, \infty]$ be given such that $f(\cdot, \xi)$ is convex, lower 
  semi-continuous and proper 
  almost surely. 
  Additionally, $f(\cdot, \xi)$ fulfills the following conditions:
  \begin{itemize}
  \item The expectation $\E_{\xi}\big[f(\cdot, \xi)\big] =: F(\cdot)$ is lower semi-continuous 
  and proper.

  \item The function $f(\cdot, \xi)$ is G\^{a}teaux differentiable almost surely on a 
  non-empty common domain $\domain{\nabla f} \subseteq H$, i.e. for all for all $v,w \in 
  \domain{\nabla f}$ the inequality $\dual{\iota \nabla f (v, \xi)}{w} = \lim_{h \to 0} \frac{f(v + 
  hw, \xi) - f(v, \xi)}{h}$ is fulfilled almost surely.

  \item There exists $m \in \N$ such that $\big(\E_{\xi}\big[ \| \nabla f(w^*,\xi)\|_{H^*}^{2^m} \big]\big)^{2^{-m}} =: \sigma < \infty$.
      
  \item For every $R > 0$ there exists $\Lr{\xi}{R} \colon \Omega \to \R$ such that
    \begin{equation*}
      \| \nabla f(u, \xi) - \nabla f(v, \xi) \|_{H^*} \le \Lr{\xi}{R} \|u - v\|
    \end{equation*}
    almost surely for all $u, v \in \domain{\nabla f}$ with $\|u \|, \|v\| \leq R$. Furthermore, there 
    exists a polynomial $P \colon \R \to \R$ of degree $2^m -2$ such that $\Lr{\xi}{R}
    \leq P(R)$ almost surely.
    
  \item There exist a random variable $M_{\xi} \colon \Omega \to \L(H)$ such that the 
    image is symmetric and a random variable $\mu_{\xi} \colon \Omega \to [0, \infty)$ 
    such that $\E_{\xi}[\mu_{\xi}] = \mu > 0$ and $\E_{\xi}[\mu_{\xi}^2] = \nu^2 < \infty$. 
    Moreover, 
    \begin{equation*}
      \dual{\nabla f(u, \xi) - \nabla f(v, \xi)}{u - v} 
      \geq \inner{M_{\xi}(u - v)}{u - v}
      \geq \mu_{\xi} \|u - v\|^2
    \end{equation*}
    is fulfilled almost surely for all $u,v \in \domain{\nabla f}$. 
    
  \end{itemize}
\end{assumption}

An immediate result of Assumption~\ref{ass:f}, is that the gradient $\nabla 
f(\cdot, \xi)$ is maximal monotone almost surely, 
see~\cite[Theorem~A]{Rockafellar.1970}. As a consequence, the resolvent (proximal 
operator)
\begin{equation*}
  T_{f, \xi} = (I + \nabla f(\cdot,\xi))^{-1}
\end{equation*}
is well-defined almost surely, see 
Lemma~\ref{lem:ExistenceMeasurable} for more details. Further, each resolvent maps
into $\domain{\nabla f}$, and as a consequence every iterate $w^k \in \domain{\nabla f}$. Finally, we may interchange expectation and differentation so that $\nabla F(w) = \E_{\xi}[\nabla f(\xi, w)]$. In our case, this can be shown via a straightforward argument based on dominated convergence similar to~\cite[Lemma 6]{RyuBoyd.2016}, but we note that it also holds in more general settings~\cite{Papageorgiou.1997,RockafellarWets.1982}.

\begin{remark}
  The idea behind the operators $M_{\xi}$ is that each $f(\cdot, \xi)$ is is allowed to be 
  only convex rather than strongly convex. However, they should be strongly convex in 
  \emph{some} directions, such that $f(\cdot, \xi)$ is strongly convex \emph{in 
  expectation}. 
  By assumption, $F$ is lower semi-continuous, proper and strongly convex, so there 
  is a minimum $w^*$ of \eqref{eq:main_problem} (c.f.~\cite[Proposition~1.4]{Barbu.2010})
  which is unique due to the strong convexity.
\end{remark}

\begin{remark}
  While the properness of $F$ needs to be verified by application-specific means, the 
  lower semi-continuity can be guaranteed on a more general level in different ways.
  If, e.g., it is additionally known that $\E_{\xi} \big[\inf_{u \in H} f(u, \xi) \big] > -\infty$
  then one can employ Fatou's lemma (\cite[Theorem~2.3.6]{Winkert.2018}) as in \cite[Lemma 
  5]{RyuBoyd.2016}, or slightly modify~\cite[Corollary 9.4]{BauschkeCombettes.2017}.  
\end{remark}

\begin{remark}
  We note that from a function analytic point of view, we are dealing with bounded 
  rather than unbounded operators $\nabla F$. However, also operators that are 
  traditionally seen as unbounded fit into the framework, given that the space $H$ 
  is chosen properly. For example, the functional $F(w) = \frac{1}{2}\int{\|\nabla 
  w\|^2}$ 
  corresponding to $\nabla F = - \Delta$, the negative Laplacian, is unbounded on $H = 
  L^2$. But if we instead choose $H = H^1_0$, then $H^* = H^{-1}$ and $\nabla 
  F$ is bounded and Lipschitz continuous. In this case, the splitting of $F(w)$ into 
  $f(w, \xi^k)$ is less obvious than in our main application, but e.g.\ (randomized) 
  domain decomposition as in ~\cite{QuarteroniValli.1999} is a natural idea. In 
  each step, an elliptic problem then has to be solved (to apply $\iota$), but 
  this can often be done very efficiently.
\end{remark}

Our main theorem states that we have sub-linear convergence of the iterates $w^k$ to 
$w^*$ measured in this expectation:

\begin{theorem} \label{theorem:main_theorem}
  Let Assumption~\ref{ass:f} be fulfilled and let $\{\xi^k\}_{k \in \N}$ be a family 
  of jointly independent random variables on $\Omega$. Then the 
  scheme~\eqref{eq:scheme} converges sub-linearly
  if the step sizes fulfill $\alpha_k = \frac{\eta}{k}$ with $\eta > 
  \frac{1}{\mu}$. In particular, the error bound
  \begin{equation*}
    \E_{k-1}\big[ \|w^k - w^* \|^2 \big] \le \frac{C}{k}
  \end{equation*}
  is fulfilled, where $C$ depends on $\|w_1 - w^* \|$, $\mu$, $\nu$, 
  $\sigma$, $\eta$ and $m$.
\end{theorem}

The proof of this theorem is given in Section~\ref{section:main_theorem}. The main idea is to acquire a contraction property of the form
\begin{equation*}
  \E_{k-1}\big[ \|w^k - w^* \|^2 \big] \le C_k \E_{k-2}\big[ \|w^{k-1} - w^* \|^2 \big] + \alpha_k^2 D,
\end{equation*}
where $C_k < 1$ and $D$ are certain constants depending on the data. Inevitably, $C_k 
\to 1$ as $k \to \infty$, but because of the chosen step size sequence this happens 
slowly enough to still guarantee the optimal rate. To reach this point, we first show two 
things: First, an a priori bound of the form $\E_{k-1}\big[ \|w^k - w^* \|^2 \big] \le C$, i.e.\ 
the scheme is always stable. Secondly, that the resolvents $T_{f, \xi}$ are contractive 
with 
\begin{equation*}
  \E_{\xi} \big[\| T_{f, \xi} u - T_{f, \xi} v \|^2\big] \le C_k\|u-v\|^2.
\end{equation*}
Similarly to~\cite{RyuBoyd.2016}, we do the latter by approximating the functions $f(\cdot, \xi)$ by convex quadratic functions $\tilde{f}(\cdot, \xi)$ for which the property is easier to verify, and then establishing a relation between the approximated and the true contraction factors. The series of lemmas in the next section is devoted to this preparatory work.

\section{Preliminaries}\label{section:preliminaries}
First, let us show that the scheme is in fact well-defined, in the sense that every 
iterate is measurable if the random variables $\{\xi^k\}_{k \in \N}$ are.

\begin{lemma} \label{lem:ExistenceMeasurable}
  Let Assumption~\ref{ass:f} be fulfilled. Further, let $\{\xi^k\}_{k \in \N}$ be a 
  family of jointly independent random variables. 
  Then for every $k \in \N$ there exists a unique mapping $w^{k+1} \colon 
  \Omega \to \domain{\nabla f}$ that fulfills 
  \eqref{eq:scheme} and is measurable with respect to the joint probability 
  distribution of $\xi^1, \ldots, \xi^k$.
\end{lemma}

\begin{proof}
  We define the mapping
  \begin{align*}
    h \colon \domain{\nabla f} \times \Omega \to H, \quad
    (u, \omega) \mapsto w^k - (I + \alpha_k  \iota  \nabla f(\cdot, \xi^k(\omega))) 
    u.
  \end{align*}
  For almost all $\omega \in \Omega$, the mapping $f(\cdot, 
  \xi^k(\omega))$ is lower semi-continuous, proper and convex. Thus, by 
  \cite[Theorem~A]{Rockafellar.1970} $\nabla f(\cdot, \xi^k(\omega))$ is maximal 
  monotone. 
  By \cite[Theorem~2.2]{Barbu.2010}, this shows that the operator $\iota^{-1} + 
  \alpha_k \nabla f(\cdot, \xi^k(\omega)) \colon \domain{\nabla f} \to H^*$ is surjective. 
  Note that the two previously cited results are stated for multi-valued operators. As we 
  are in a more regular setting, the sub-differential of $f(\cdot, \xi^k(\omega))$ only 
  consists of a single element at each point. 
  Therefore, it is possible to apply these multi-valued results also in our setting and 
  interpret the appearing operators as single-valued.
  Furthermore, due to the monotonicity of $\nabla f(\cdot, \xi^k(\omega))$ it 
  follows that for $u,v \in \domain{\nabla f}$ 
  \begin{align*}
    \dual{\big(\iota^{-1} + \alpha_k 
      \nabla f(\cdot, \xi^k(\omega))\big) u - \big(\iota^{-1} + \alpha_k 
      \nabla f(\cdot, \xi^k(\omega))\big) v}{u -v }
    \geq \|u-v\|^2
  \end{align*}
  which implies
  \begin{align*}
    \big\|\big(\iota^{-1} + \alpha_k 
      \nabla f(\cdot, \xi^k(\omega))\big) u - \big(\iota^{-1} + \alpha_k 
      \nabla f(\cdot, \xi^k(\omega))\big) v \big\|
    \geq \|u-v\|.
  \end{align*}
  This verifies that $I + \alpha_k \iota \nabla f(\cdot, \xi^k(\omega))$ is injective. As we 
  have proved that the operator is both injective and surjective, it is, in particular, bijective.
  Therefore, there exists a unique element $w^{k+1}(\omega)$ such that
  \begin{align*}
    h(w^{k+1}(\omega), \omega) = w^k - (I + \alpha_k  \iota \nabla f(\cdot, 
    \xi^k(\omega))) w^{k+1}(\omega) = 0.
  \end{align*}
  We can now apply \cite[Lemma~2.1.4]{EisenmannPhD.2019} or \cite[Lemma~    
  4.3]{EisenmannEtAl.2019} and obtain that $\omega \mapsto w^{k+1}(\omega)$ 
  is measurable.
\end{proof}

Proving that the scheme is always stable is relatively straightforward, as shown in the 
next lemma. With some extra effort, we also get stability in stronger norms, i.e.\ we can 
bound not only $\E_{k}\big[\|w^{k+1} - w^*\|^2\big]$ but also higher moments 
$\E_{k}\big[\|w^{k+1} - w^*\|^{2^m}\big]$, $m \in \N$. This will be important since we only 
have the weaker local Lipschitz continuity stated in Assumption~\ref{ass:f} rather than 
global Lipschitz continuity.

\begin{lemma}\label{lem:apriori}
  Let Assumption~\ref{ass:f} be fulfilled, and suppose that 
  $\sum_{k=1}^{\infty}{\alpha_k^2} < \infty$.  Then there exists a
  constant $D \geq 0$ depending only on $\|w_1 - w^*\|$, 
  $\sum_{k=1}^{\infty}{\alpha_k^2}$ and $\sigma$, such that
  \begin{equation*}
    \E_{k}\big[\|w^{k+1} - w^*\|^{2^m}\big] \leq D
  \end{equation*}
  for all $k \in \N$.
\end{lemma}

\begin{proof}
  Within the proof, we abbreviate the function $f(\cdot, \xi^k)$ by $f_k$, $k \in 
  \N$.
  First, we consider first the case $m = 1$.  Recall the identity $\inner{a - b}{a} = 
  \frac{1}{2} \big(\|a\|^2 - \|b\|^2 + \|a-b\|^2\big)$, $a,b \in H$.
  We write the scheme as
  \begin{align*}
    w^{k+1} - w^k + \alpha_k \iota \nabla f_k(w^{k+1}) = 0,
  \end{align*}
  subtract $\alpha_k \iota \nabla f_k(w^{*})$ from both sides, multiply by two and test it 
  with $w^{k+1} - w^*$ to obtain
  \begin{align*}
    &\|w^{k+1} - w^*\|^2 - \|w^k - w^*\|^2 + \|w^{k+1} -w^k\|^2 \\
    &\quad + 2 \alpha_k \inner{\iota \nabla f_k(w^{k+1}) - \iota \nabla 
    f_k(w^*)}{w^{k+1} - w^*} \\
    &= - 2 \alpha_k \inner{\iota \nabla f_k(w^*)}{w^{k+1} - w^*}.
  \end{align*}
  For the right-hand side, we have by Young's inequality that
  \begin{align*}
    &- 2 \alpha_k \inner{ \iota \nabla f_k(w^*)}{w^{k+1} - w^*} \\
    &= - 2 \alpha_k \dual{\nabla f_k(w^*)}{w^{k+1} - w^k} - 2 \alpha_k 
    \dual{\nabla f_k(w^*)}{w^k - w^*}\\
    &\leq 2 \alpha_k \|\nabla f_k(w^*)\|_{H^*} \|w^{k+1} - w^k\| - 2 
    \alpha_k \dual{\nabla f_k(w^*)}{w^k - w^*}\\
    &\leq \alpha_k^2 \|\nabla f_k(w^*)\|_{H^*}^2 + \|w^{k+1} - 
    w^k\|^2 - 2 \alpha_k \dual{\nabla f_k(w^*)}{w^k - w^*}.
  \end{align*}
  Together with the monotonicity condition, it then follows that
  \begin{align}\label{eq:aprioriProof1}
    \begin{split}
      \|w^{k+1} - w^*\|^2 - \|w^k - w^*\|^2
      &\leq \alpha_k^2 \|\nabla f_k(w^*)\|_{H^*}^2- 2 \alpha_k 
      \dual{\nabla f_k(w^*)}{w^k - w^*}.
    \end{split}
  \end{align}
  Since $w^k - w^*$ is independent of $\xi^k$ and $\E_{\xi^k}[\nabla f_k(w^*)] = \nabla F(w^*) = 0$,
  taking the expectation $\E_{\xi^k}$ thus leads to the following bound:
  \begin{align*}
    \E_{\xi^k}\big[\|w^{k+1} - w^*\|^2\big]
    \leq \|w^k - w^*\|^2 + \alpha_k^2 \sigma^2.
  \end{align*}
  Repeating this argument, we obtain that
  \begin{align*}
    \E_{k}\big[\|w^{k+1} - w^*\|^2\big]
    \leq \|w_1 - w^*\|^2 + \sigma^2 \sum_{j = 1}^{k}\alpha_{j}^2.
  \end{align*}
  In order to find the higher moment bound, we recall \eqref{eq:aprioriProof1}.
  We then follow a similar idea as in \cite[Lemma~3.1]{BrzezniakEtAll.2013}, where
  we multiply this inequality with $\|w^{k+1} - w^*\|^2$ and use the 
  identity $(a - b)a = \frac{1}{2} \big(|a|^2 - |b|^2 + |a-b|^2\big)$ for $a,b \in 
  \R$. It then follows that
  \begin{align*}
    &\|w^{k+1} - w^*\|^4 - \|w^k - w^*\|^4 + \big| \|w^{k+1} - w^*\|^2 - \|w^k 
    - w^*\|^2 \big|^2\\
    &\leq \Big( \alpha_k^2 \|\nabla f_k(w^*)\|_{H^*}^2- 2 
    \alpha_k \dual{\nabla f_k(w^*)}{w^k - w^*}\Big) \|w^{k+1} - w^*\|^2\\
    &\leq \Big( \alpha_k^2 \|\nabla f_k(w^*)\|_{H^*}^2- 2 
    \alpha_k \dual{\nabla f_k(w^*)}{w^k - w^*} \Big)\\
    &\quad \times \Big(\|w^k - w^*\|^2 + \alpha_k^2 \|\nabla 
    f_k(w^*)\|_{H^*}^2- 2 \alpha_k 
    \dual{\nabla f_k(w^*)}{w^k - w^*}\Big)\\
    &\leq \alpha_k^2 \|w^k - w^*\|^2 \|\nabla f_k(w^*)\|_{H^*}^2- 2 
    \alpha_k \|w^k - w^*\|^2 \dual{\nabla f_k(w^*)}{w^k - w^*} \\
    &\quad + \alpha_k^4 \|\nabla f_k(w^*)\|_{H^*}^4- 4 
    \alpha_k^3 \|\nabla f_k(w^*)\|_{H^*}^2 \dual{\nabla f_k(w^*)}{w^k - 
    w^*}\\
    &\quad + 4 \alpha_k^2 \big(\dual{\nabla f_k(w^*)}{w^k - w^*}\big)^2.
  \end{align*}
  Applying Young's inequality to the first and fourth term of the previous row then implies 
  that
  \begin{align*}
    &\|w^{k+1} - w^*\|^4 - \|w^k - w^*\|^4\\
    &\leq \frac{\alpha_k^2}{2}\|w^k - w^*\|^4 - 2 
    \alpha_k \|w^k - w^*\|^2 \dual{\nabla f_k(w^*)}{w^k - w^*} \\
    &\quad + \Big(3\alpha_k^4 + \frac{\alpha_k^2}{2}\Big) \|\nabla 
    f_k(w^*)\|_{H^*}^4 
    + 6 \alpha_k^2 \|\nabla f_k(w^*)\|_{H^*}^2 \|w^k - w^*\|^2\\
    &\leq \frac{\alpha_k^2}{2}\|w^k - w^*\|^4 - 2 
    \alpha_k \|w^k - w^*\|^2 \dual{\nabla f_k(w^*)}{w^k - w^*} \\
    &\quad + \Big(3\alpha_k^4 + \frac{\alpha_k^2}{2}\Big) \|\nabla 
    f_k(w^*)\|_{H^*}^4 
    + 3 \alpha_k^2 \|\nabla f_k(w^*)\|_{H^*}^4 
    + 3 \alpha_k^2 \|w^k - w^*\|^4\\
    &\leq \frac{7\alpha_k^2}{2}\|w^k - w^*\|^4 - 2 
    \alpha_k \|w^k - w^*\|^2 \dual{\nabla f_k(w^*)}{w^k - w^*} \\
    &\quad + \Big(3\alpha_k^4 + \frac{7\alpha_k^2}{2}\Big) \|\nabla 
    f_k(w^*)\|_{H^*}^4.
  \end{align*}
  Summing up from $j=1$ to $k$ and taking the expectation $\E_{k}$, yields
  \begin{align*}
    &\E_{k} \big[\|w^{k+1} - w^*\|^4\big] \\
    &\leq \|w_1 - w^*\|^4 + \sum_{j = 1}^{k} \frac{7\alpha_j^2}{2} \E_{j-1} 
    \big[\|w^j - w^*\|^4 \big]
    + \sigma^4 \sum_{j=1}^{k} \Big(3\alpha_j^4 + \frac{7\alpha_j^2}{2}\Big).
  \end{align*}
  We then apply the discrete Gr\"{o}nwall inequality for sums (see, e.g., 
  \cite{Clark.1987}) which shows that
  \begin{align*}
    \E_k \big[\|w^{k+1} - w^*\|^4\big] 
    \leq \Big(\|w_1 - w^*\|^4 
    + \sigma^4 \sum_{j=1}^{k} \Big(3\alpha_j^4 + \frac{7\alpha_j^2}{2}\Big)\Big)
    \exptextB{\frac{7}{2}\sum_{j = 1}^{k} \alpha_j^2}.
  \end{align*}
  For the next higher bound $\E_k \big[\|w^{k+1} - w^*\|^8\big]$, we recall that
  \begin{align*}
    &\|w^{k+1} - w^*\|^4 - \|w^k - w^*\|^4\\
    &\leq \frac{7\alpha_k^2}{2}\|w^k - w^*\|^4 - 2 
    \alpha_k \|w^k - w^*\|^2 \dual{\nabla f_k(w^*)}{w^k - w^*} \\
    &\quad + \Big(3\alpha_k^4 + \frac{7\alpha_k^2}{2}\Big) \|\nabla 
    f_k(w^*)\|_{H^*}^4,
  \end{align*}
  which we can multiply with $\|w^{k+1} - w^*\|^4$ in order to follow the same strategy as 
  before.
  Following this approach, we find bounds for $\E_k \big[\|w^{k+1} - w^*\|^{2^m}\big]$ 
  recursively for all $m \in \N$.
\end{proof}

\begin{remark}\label{rem:apriori}
  In particular, Lemma~\ref{lem:apriori} implies that there exists a constant $D$ 
  depending on $\|w_1 - w^*\|$, $\sum_{k=1}^{\infty}{\alpha_k^2}$ and $\sigma$ such that 
  \begin{equation*}
  \E_{k}\big[\|w^{k+1} - w^*\|^{p}\big] \leq D
  \end{equation*}
for all $p \leq 2^m$ and $k \in \N$.
\end{remark}

For the further  analysis, we now introduce the function $\tilde{f}(\cdot,\xi) 
\colon H \times \Omega \to (-\infty,\infty]$ given by
\begin{align}\label{eq:defTildeF}
  \tilde{f}(u,\xi)=f(u_0,\xi) + \dual{ \nabla f(u_0,\xi)}{u - u_0}+ \frac{1}{2} 
  \inner{M_{\xi}(u - u_0)}{u - u_0},
\end{align}
where $u_0 \in \domain{\nabla f}$ is a fixed parameter.
This mapping is a convex approximation of $f$. Furthermore, we define the 
function $\tilde{r}(\cdot,\xi) \colon H \times \Omega \to (-\infty,\infty]$ given by
\begin{align}\label{eq:defTildeR}
  \tilde{r}(u,\xi) = f(u,\xi) - \tilde{f}(u,\xi).
\end{align} 
Their gradients $\nabla \tilde{f}(\cdot,\xi) \colon H \times \Omega \to H^*$ 
and $\nabla \tilde{r}(\cdot,\xi) \colon \domain{\nabla f} \times \Omega \to H^*$ can be 
stated as
\begin{align*}
  \nabla \tilde{f}(u,\xi) &= \nabla f(u_0,\xi) + \inner{M_{\xi}(u - u_0)}{\cdot}, \quad u \in H,\\
  \nabla \tilde{r}(u,\xi) &= \nabla f(u,\xi) - \nabla f(u_0,\xi) - \inner{M_{\xi}(u - 
  u_0)}{\cdot}, \quad u \in \domain{\nabla f}
\end{align*}
 almost surely.

\begin{lemma}\label{lem:r}
  The function $\tilde{r}(\cdot, \xi)$ defined in \eqref{eq:defTildeR} is convex almost surely, 
  i.e., it fulfills $\tilde{r}(u,\xi) \geq \tilde{r}(v,\xi) + \dual{\nabla \tilde{r}(v,\xi)}{ u - v}$ 
  for all $u,v \in \domain{\nabla f}$ almost surely. As a consequence, the gradient $\nabla 
  \tilde{r}(\cdot,\xi)$ is monotone almost surely.
\end{lemma}

\begin{proof}
  In the following proof, let us omit $\xi$ for simplicity and let $u,v \in \domain{\nabla f}$ 
  be given.
  Due to the monotonicity property of $\nabla f$ stated in Assumption~\ref{ass:f}, it follows that
  \begin{align*}
    f(u) \geq f(v) + \dual{\nabla f(v)}{ u - v} + \frac{1}{2}\inner{M(u - 
    v)}{u - v}.
  \end{align*}
  For the function $\tilde{f}$ we can write
  \begin{align*}
    \tilde{f}(u) &= f(u_0) + \dual{\nabla f(u_0)}{u - u_0}+ \frac{1}{2} 
    \inner{M(u -
      u_0)}{u - u_0},\\
    \nabla \tilde{f}(u) &= \nabla f(u_0) + \inner{M(u -u_0)}{\cdot} 
    \quad \text{and} \quad
    \nabla^2 \tilde{f}(u) = M.
  \end{align*}
  All further derivatives are zero. Thus, we can use a Taylor expansion around $v$ to write
  \begin{align*}
    \tilde{f}(u) 
    &= \tilde{f}(v) + \dual{\nabla \tilde{f}(v)}{ u - v} + \frac{1}{2}\inner{M(u - 
    v)}{u - v}.
  \end{align*}
  It then follows that
  \begin{align*}
    \tilde{r}(u) 
    &\geq f(v) + \dual{\nabla f(v)}{ u - v} + \frac{1}{2}\inner{M(u - v)}{u - v}\\
    &\quad - \big(\tilde{f}(v) + \dual{\nabla \tilde{f}(v)}{ u - v} + 
    \frac{1}{2}\inner{M(u -
    v)}{u - v}\big)\\
    &= \tilde{r}(v) + \dual{\nabla \tilde{r}(v)}{ u - v}.
  \end{align*}
  By \cite[Proposition 25.10]{ZeidlerIIB.1990}, it follows that $\nabla \tilde{r}$ is monotone.
\end{proof}

The following lemma demonstrates that the resolvents $T_{\tilde{f}, \xi}$ and certain
perturbations of them are well-defined. Furthermore, we will provide a more explicit 
formula for such resolvents.

\begin{lemma}\label{lem:welldef_resolvents}
  Let Assumption~\ref{ass:f} be fulfilled and let $\tilde{f}(\cdot, \xi)$ be defined as in 
  \eqref{eq:defTildeF}. Then the operator 
  \begin{align*}
  T_{\tilde{f}, \xi} = (I + \iota \nabla \tilde{f}(\cdot,\xi))^{-1} \colon H \times \Omega 
  \to H
  \end{align*} 
  is well-defined.
  If a function $r(\cdot, \xi) \colon H \times \Omega \to (-\infty, \infty]$ is G\^{a}teaux 
  differentiable with the common domain $\domain{\nabla r} = \domain{\nabla f}$, lower semi-continuous, convex and proper almost surely, then
  \begin{align*}
    T_{\tilde{f}+ r, \xi} = (I + \iota \nabla \tilde{f}(\cdot,\xi) + \iota \nabla 
    r(\cdot,\xi))^{-1} \colon H \times \Omega \to \domain{\nabla f}
  \end{align*}
  is well-defined.
  
  If there exist $Q_{\xi} \colon \domain{\nabla f} \times \Omega \to H^*$ and $z_{\xi} \colon \Omega \to 
  H^*$ such that $\nabla r (u,\xi) = Q_{\xi} u + z_{\xi}$ then the resolvent can be 
  represented by
  \begin{align*}
    T_{\tilde{f} + r, \xi} u 
    = (I + M_{\xi} + \iota Q_{\xi})^{-1} \big(u - \iota \nabla f(u_0,\xi) + M_{\xi}u_0 - \iota 
    z_{\xi}\big).
  \end{align*}
\end{lemma}

\begin{proof}
  For simplicity, let us omit $\xi$ again.
  In order to prove that $T_{\tilde{f}}$ and $T_{\tilde{f} + r}$ are well-defined, we 
  can apply \cite[Theorem~A]{Rockafellar.1970} and 
  \cite[Theorem~2.2]{Barbu.2010} analogously to the argumentation in the proof of 
  Lemma~\ref{lem:ExistenceMeasurable}.
  
  Assuming that $\nabla r (u) = Q u + z$, we find an explicit representation for 
  $T_{\tilde{f} + r}$. To this end, for $v \in H$, consider
  \begin{align*}
    (I + \iota \nabla \tilde{f} + \iota \nabla r)^{-1} v = T_{\tilde{f} + r} v =: u
    \in \domain{\nabla f}.
  \end{align*}
  Then it follows that
  \begin{align*}
    v = (I + \iota \nabla \tilde{f} + \iota \nabla r) u
    = (I + M + \iota Q) u + \iota \nabla f(u_0) - Mu_0 + \iota z.
  \end{align*}
  Rearranging the terms, yields
  \begin{align*}
    T_{\tilde{f} + r} v 
    &= (I + M + \iota Q)^{-1} \big(v - \iota \nabla f(u_0) + Mu_0 - \iota z\big).
  \end{align*}
\end{proof}

Next, we will show that the contraction factors of $T_{f, \xi}$ and $T_{\tilde{f}, \xi}$ are 
related. For this, we need the following basic identities and some stronger inequalities that 
hold for symmetric positive operators on $H$.

\begin{lemma}\label{lem:basic_identities}
  Let Assumption~\ref{ass:f} be satisfied and let $\tilde{f}(\cdot, \xi)$ and $\tilde{r}(\cdot, 
  \xi)$ be given 
  as in \eqref{eq:defTildeF} and \eqref{eq:defTildeR}, respectively. Then the 
  identities 
  \begin{align*}
    \iota \nabla f  (T_{f, \xi}, \xi) = I - T_{f,\xi}
    \quad 
    \text{and}
    \quad
    \iota \nabla \tilde{f} (T_{f,\xi},\xi) + T_{f,\xi} - I
    = - \iota \nabla \tilde{r} (T_{f,\xi},\xi)
  \end{align*}
  are fulfilled almost surely.
\end{lemma}

\begin{proof}
  By the definition of $T_{f,\xi}$, we have that
  \begin{align*}
    T_{f,\xi} + \iota \nabla f  (T_{f,\xi},\xi) = (I + \iota \nabla f(\cdot,\xi) ) T_{f,\xi} = I,
  \end{align*}
  from which the first claim follows immediately.
  The second identity then follows from
  \begin{align*}
    \iota \nabla \tilde{f} (T_{f,\xi},\xi) + T_{f,\xi} - I
    = \iota \nabla \tilde{f} (T_{f,\xi},\xi) - \iota \nabla f (T_{f,\xi},\xi)
    = - \iota \nabla \tilde{r} (T_{f,\xi},\xi) .
  \end{align*}
\end{proof}

As a consequence of Lemma~\ref{lem:basic_identities} we have the following
basic inequalities:

\begin{lemma}\label{lem:basic_inequalities}
  Let Assumption~\ref{ass:f} be satisfied. It then follows that
  \begin{equation*}
    \|T_{f,\xi} u - u\| \le \|\nabla f(u,\xi)\|_{H^*}
  \end{equation*}
  almost surely for every $u \in \domain{\nabla f}$. Additionally, if for 
  $R >0$ the bound $\|u\| + \|\nabla f(u,\xi)\| \le R$ holds true almost surely, 
  then the second-order estimate
  \begin{equation*}
    \|\iota^{-1}(T_{f,\xi} u - u) + \nabla f(u,\xi)\|_{H^*} \le \Lr{\xi}{R}\|\nabla 
    f(u,\xi)\|_{H^*}
  \end{equation*}
  is fulfilled almost surely.
\end{lemma}

\begin{proof}
  In order to shorten the notation, we omit the $\xi$ in the following proof and let $u$ be in 
  $\domain{\nabla f}$.
  For the first inequality, we note that since $\nabla f$ is monotone, we have
  \begin{equation*}
    \dual{\nabla f (T_f u) - \nabla f(u)}{T_f u - u} \ge 0.
  \end{equation*}
  Thus, by the first identity in Lemma~\ref{lem:basic_identities}, it follows that
  \begin{align*}
    \dual{-\nabla f(u)}{T_f u - u} &=
    \dual{\nabla f(T_fu) - \nabla f(u)}{T_f u - u} - \dual{\nabla f (T_f 
      )}{T_f u - u}\\
    &\ge \dual{\iota^{-1}(T_f u - u)}{T_f u - u}\\
    &= \inner{T_f u - u }{T_f u - u}
    = \|T_f u - u \|^2.
  \end{align*}
  But by the Cauchy-Schwarz inequality, we also have
  \begin{equation*}
    \dual{-\nabla f(u)}{T_f u - u} \le \|\nabla f(u)\|_{H^*} \|T_f u - u \|,
  \end{equation*}
  which in combination with the previous inequality proves the first claim.
  
  The second inequality follows from the first part of this lemma. Because
  \begin{align*}
    \|T_f u\| 
    \leq \|T_f u - u\| + \|u\| 
    \leq \|\nabla f(u) \|_{H^*} + \|u\|,
  \end{align*}
  both $u$ and $T_f u$ are in a ball of radius $R$. Thus, we obtain 
  \begin{align*}
    \|\iota^{-1} (T_f u - u) + \nabla f(u)\|_{H^*} 
    &= \|\nabla f(u) - \nabla f(T_f u)\|_{H^*}  \\
    &\le \Lr{}{R} \|u - T_f u\| 
    \le \Lr{}{R} \|\nabla f(u)\|_{H^*}.
  \end{align*}
\end{proof}

\begin{lemma}\label{lem:positive_operator_inequalities}
  Let $Q, S \in \L(H)$ be symmetric operators. Then the following holds:
  \begin{itemize}
    \item If $Q$ is invertible and $S$ and $Q^{-1}$ are strictly positive. Then $(Q 
    + S)^{-1} < Q^{-1}$. If $S$ is only positive, then $(Q + S)^{-1} \leq Q^{-1}$.
    \item If $Q$ is a positive and contractive operator, i.e. $\|Qu\| \le \|u\|$ for all 
    $u \in H$, then it follows that $\|Qu\|^2 \leq \inner{Qu}{u}$ for all $u \in H$.
    \item If $Q$ is a strongly positive invertible operator, such that there exists 
    $\beta > 0$ with $\inner{Q u}{u} \geq \beta \|u\|^2$ for all $u \in H$, then 
    $\|Q u \| \geq \beta \|u\|$ for all $u \in H$ and 
    $\|Q^{-1}\|_{\L(H)} \leq \frac{1}{\beta}$.
  \end{itemize}
\end{lemma}

\begin{proof}
We start by expressing $(Q + S)^{-1}$ in terms of $Q^{-1}$ and $S$, similar to 
the Sherman-Morrison-Woodbury formula for matrices~\cite{Hager.1989}. First 
observe that the operator $(I + Q^{-1}S)^{-1} \in \L(H)$ by 
e.g.~\cite[Lemma 2A.1]{LasieckaTriggiani.2000}.
Then, since
\begin{align*}
  &\Big(Q^{-1} - Q^{-1}S\big(I + Q^{-1}S \big)^{-1}Q^{-1}\Big) (Q+S) \\
&\quad = I + Q^{-1}S - Q^{-1}S\big(I + Q^{-1}S \big)^{-1}\big(I + Q^{-1}S \big) 
= I
\end{align*}
and
\begin{align*}
  &(Q+S)\Big(Q^{-1} - Q^{-1}S\big(I + Q^{-1}S \big)^{-1}Q^{-1}\Big)  \\
&\quad = I + SQ^{-1} - S\big(I + Q^{-1}S \big)\big(I + Q^{-1}S \big)^{-1}Q^{-1} 
= I,
\end{align*}
we find that
\begin{align*}
  (Q + S)^{-1} = Q^{-1} - Q^{-1}S\big(I + Q^{-1}S \big)^{-1}Q^{-1}.
\end{align*}
Since $Q^{-1}$ is symmetric, we see that $(Q + S)^{-1} < Q^{-1}$ if and only if 
$S\big(I + Q^{-1}S \big)^{-1}$ is strictly positive. But this is true, as we see 
from 
the change of variables $z = (I + Q^{-1}S)^{-1} u$. Because then
\begin{align*}
  \innerB{S\big(I + Q^{-1}S \big)^{-1} u}{u} &= \innerb{Sz}{z + Q^{-1}Sz} 
  = \innerb{Sz}{z} + \innerb{Q^{-1}Sz}{Sz} > 0
\end{align*}
for any $u \in H$, $u \neq 0$, since $S$ and $Q^{-1}$ are strictly positive. If $S$ 
is only positive, it follows analogously that $\innerb{S\big(I + Q^{-1}S 
\big)^{-1} u}{u} \geq 0$.

  In order to prove the second statement, we use the fact that there exists a 
  unique symmetric and positive square root $Q^{\nicefrac{1}{2}} \in \L(H)$ such 
  that $Q = 
  Q^{\nicefrac{1}{2}}Q^{\nicefrac{1}{2}}$. Since $\|Q\| = \sup_{x \in H} \inner{Q 
  x}{x}  = \sup_{x \in H} \inner{Q^{\frac{1}{2}} x}{Q^{\frac{1}{2}} x} 
  = \|Q^{\nicefrac{1}{2}}\|^2$, also $Q^{\nicefrac{1}{2}}$ is contractive. Thus 
  \begin{align*}
    \|Q u\|^2 = \|Q^{\nicefrac{1}{2}}Q^{\nicefrac{1}{2}}u\|^2 \le  
    \|Q^{\nicefrac{1}{2}}u\|^2 = 
    \inner{Q^{\nicefrac{1}{2}}u}{Q^{\nicefrac{1}{2}}u} = \inner{Qu}{u}.
  \end{align*}
  
  Now, we prove the third statement. First we notice that $\inner{Qu}{u} \geq \beta 
  \|u\|^2$ and $\inner{Qu}{u} \leq \|Qu\| \|u\|$ imply that $\|Qu\| \geq \beta 
  \|u\|$ for all $u \in H$. Substituting $v = Q^{-1}u$, then shows $\|v\| \geq 
  \beta \|Q^{-1}v \|$, which proves the final claim. 
\end{proof}

\begin{lemma} \label{lem:same_contraction_factors}
  Let Assumption~\ref{ass:f} be fulfilled and let $\tilde{f}(\cdot, \xi)$ be given as in 
  \eqref{eq:defTildeF}. Then
  \begin{align*}
    \E_{\xi} \Bigg[\frac{\| T_{f, \xi} u - T_{f, \xi} v \|^2}{\|u-v\|^2} \Bigg]
    \leq \Bigg( \E_{\xi} \Bigg[\frac{\| T_{\tilde{f}, \xi} u - T_{\tilde{f}, \xi} v \|^2}{\|u-v\|^2} 
    \Bigg]\Bigg)^{\nicefrac{1}{2}}
  \end{align*}
  holds for every $u,v \in H$.
\end{lemma}

\begin{proof}
  For better readability, we once again omit $\xi$ where there is no risk of confusion.
  For $u,v \in \domain{\nabla f}$ and $\varepsilon >0$, we approximate the function $\tilde{r}(\cdot, \xi)$ 
  defined in \eqref{eq:defTildeR} by 
  \begin{align*}
    \tilde{r}_{\varepsilon}(\cdot, \xi) \colon H \times \Omega \to (-\infty,\infty], \quad
    \tilde{r}_{\varepsilon} (z,\xi) = \dual{\nabla \tilde{r}(T_f u,\xi)}{ z} +
    \frac{\big(\dual{v_{\varepsilon}}{z - T_f u}\big)^2}{2 a_{\varepsilon}},
  \end{align*}
  where
  \begin{align*}
    v_{\varepsilon} = - \nabla \tilde{r}(T_f u) + \nabla \tilde{r}(T_f v) + \varepsilon 
    \iota^{-1}
    (T_f v - T_f u) \in H
    \quad \text{and} \quad
    a_{\varepsilon}  = \dual{v_{\varepsilon}}{T_f v - T_f u}.
  \end{align*}
  As we can write
  \begin{align*}
  a_{\varepsilon}
  &= \dual{- \nabla \tilde{r}(T_f u) + \nabla \tilde{r}(T_f v) + \varepsilon 
    \iota^{-1} (T_f 
    v - 
    T_f u)}{T_f v-T_f u}\\
  &= \dual{\nabla \tilde{r}(T_f u) - \nabla \tilde{r}(T_f v) }{T_f u-T_f v} + 
  \varepsilon \inner{ T_f v - T_f u}{T_f v-T_f u}\\
  &\geq \varepsilon \| T_f v - T_f u \|^2 > 0,
  \end{align*}
  $\tilde{r}_{\varepsilon}$ is well-defined.
  The derivative is given by $\nabla \tilde{r}_{\varepsilon}(\cdot,\xi) \colon H
  \times \Omega \to H^*$,
  \begin{align*}
    \nabla \tilde{r}_{\varepsilon} (z) 
    = \nabla \tilde{r}(T_f u) 
    + \frac{\dual{v_{\varepsilon}}{z - T_f u}}{a_{\varepsilon}} v_{\varepsilon}
    = \frac{\dual{v_{\varepsilon}}{z}}{a_{\varepsilon}} v_{\varepsilon} + \nabla 
    \tilde{r}(T_f 
    u) 
    - \frac{\dual{v_{\varepsilon}}{T_f u}}{a_{\varepsilon}} v_{\varepsilon}.
  \end{align*}
  This function $\nabla \tilde{r}_{\varepsilon}$ is an interpolation 
  between the points
  \begin{align*}
    \nabla \tilde{r}_{\varepsilon} (T_f u) &= \nabla \tilde{r}(T_f u) \quad \text{and}\\\\
    \nabla \tilde{r}_{\varepsilon} (T_f v) &= \nabla \tilde{r}(T_f u)
    + \frac{\dual{v_{\varepsilon}}{T_f v - T_f u}}{a_{\varepsilon}} v_{\varepsilon}\\
    &= \nabla \tilde{r}(T_f u)
    + \frac{\dual{v_{\varepsilon}}{T_f v - T_f u}}{\dual{v_{\varepsilon}}{T_f v - T_f
        u}} v_{\varepsilon}\\
    &= \nabla \tilde{r}(T_f u)
    + \big(- \nabla \tilde{r}(T_f u) + \nabla \tilde{r}(T_f v) + \varepsilon \iota^{-1} 
    (T_f v 
    - T_f 
    u)\big)\\
    &= \nabla \tilde{r}(T_f v) + \varepsilon \iota^{-1} (T_f v - T_f u).
  \end{align*}
  Furthermore, since $T_{\tilde{f} + \tilde{r}_{\varepsilon}} = (I + \iota \nabla 
  \tilde{f} + \iota \nabla \tilde{r}_{\varepsilon})^{-1}$, it 
  follows that
  \begin{align*}
    (I + \iota \nabla \tilde{f} + \iota \nabla 
    \tilde{r}_{\varepsilon} ) T_f u
    &= T_f u + \iota\nabla \tilde{f}(T_f u) + \iota\nabla \tilde{r} (T_f u)\\
    &= T_f u + \iota\nabla f(T_f u) 
    = (I + \iota\nabla f) T_f u
    = u
  \end{align*}
  and therefore
  \begin{align*}
    T_f u
    = (I + \iota \nabla \tilde{f} + \iota \nabla 
    \tilde{r}_{\varepsilon} )^{-1} u
    = T_{\tilde{f} + \tilde{r}_{\varepsilon}} u.
  \end{align*}
  Applying Lemma~\ref{lem:basic_identities}, we find that
  \begin{align*}
      &(I + \iota \nabla \tilde{f} + \iota \nabla 
      \tilde{r}_{\varepsilon} ) T_f v\\
      &= T_f v + \iota \nabla \tilde{f}(T_f v) + \iota \nabla \tilde{r} (T_f v) + 
      \varepsilon (T_f v - T_f u)\\
      &= T_f v + \iota \nabla f(T_f v) + \varepsilon (T_f v - T_f u)= v + 
      \varepsilon 
      (T_f v - T_f u).
  \end{align*}
  This shows that
  \begin{align}\label{eq:proofContraction}
    T_f v
    = (I + \iota \nabla \tilde{f} + \iota \nabla 
    \tilde{r}_{\varepsilon} )^{-1}(v + 
    \varepsilon (T_f v - T_f u))
    = T_{\tilde{f} + \tilde{r}_{\varepsilon}} (v + \varepsilon (T_f v - T_f u)).
  \end{align}
  Using the explicit representation of $T_{\tilde{f} + \tilde{r}_{\varepsilon}}$ from 
  Lemma~\ref{lem:welldef_resolvents}, it follows that
  \begin{align*}
    T_{\tilde{f} + \tilde{r}_{\varepsilon}} z 
    &= \Big(I + M + \iota \Big(\frac{\dual{v_{\varepsilon}}{\cdot}}{a_{\varepsilon}} 
    v_{\varepsilon}\Big) \Big)^{-1} \Big(z - \iota \nabla f(u_0)\\
    &\hspace{4cm}  + Mu_0 - \iota 
    \Big(\nabla \tilde{r}(T_f u) - \frac{\dual{v_{\varepsilon}}{T_f u}}{a_{\varepsilon}} 
    v_{\varepsilon}\Big) \Big).
  \end{align*}
  Therefore, we have
  \begin{align*}
    &\| T_{\tilde{f} + \tilde{r}_{\varepsilon}} v - T_{\tilde{f} + 
    \tilde{r}_{\varepsilon}} (v + \varepsilon (T_f v - T_f u))\|\\
    &\leq \Big\| \Big(I + M + \iota 
    \Big(\frac{\dual{v_{\varepsilon}}{\cdot}}{a_{\varepsilon}} 
    v_{\varepsilon}\Big) \Big)^{-1} \Big\|_{\L(H)}
    \| v - v  - \varepsilon (T_f v - T_f u) \|\\
    &\leq \varepsilon \| T_f v - T_f u \| \to 0
    \quad \text{ as } \varepsilon \to 0,
  \end{align*}
  since
  \begin{align*}
    \innerB{\Big(I + M + \iota 
    \Big(\frac{\dual{v_{\varepsilon}}{\cdot}}{a_{\varepsilon}} 
      v_{\varepsilon}\Big)\Big) u}{u} \geq \|u\|^2
  \end{align*}
means that  we can apply Lemma~\ref{lem:positive_operator_inequalities}.
  Thus, this shows that $T_f u= T_{f + \tilde{r}_{\varepsilon}} u$ and $T_f v = 
  \lim_{\varepsilon \to 0} T_{\tilde{f} + \tilde{r}_{\varepsilon}} v$.
  Further, we can state an explicit representation for $T_{\tilde{f}}$ using 
  Lemma~\ref{lem:welldef_resolvents} given by
  \begin{align*}
    T_{\tilde{f}} z = (I + \iota \nabla \tilde{f} )^{-1} z = (I + M)^{-1} 
    \big(z - \iota \nabla f(u_0) + Mu_0\big).
  \end{align*}
  For $n = \frac{u-v}{\|u-v\|}$ with $\|n\| = 1$, we obtain using
  Lemma~\ref{lem:positive_operator_inequalities}
  \begin{align*}
    \frac{\| T_{\tilde{f}} u - T_{\tilde{f}} v \|}{\|u - v \|}
    &= \| (I + M)^{-1}n \|\\
    &\geq \inner{(I + M)^{-1} n}{n} \\
    &\geq \innerB{\Big(I + M + \iota 
    \Big(\frac{\dual{v_{\varepsilon}}{\cdot}}{a_{\varepsilon}}
      v_{\varepsilon}\Big) \Big)^{-1} n}{n} \\
    &\geq \Big \|\Big(I + M + \iota 
    \Big(\frac{\dual{v_{\varepsilon}}{\cdot}}{a_{\varepsilon}}
    v_{\varepsilon}\Big) \Big)^{-1} n \Big\|^2 \\
    &= \frac{\| T_{\tilde{f} + \tilde{r}_{\varepsilon} } u - T_{\tilde{f} +
    \tilde{r}_{\varepsilon}} v \|^2}{\|u - v \|^2}
    \to \frac{\| T_f u - T_f v \|^2}{\|u - v \|^2} \quad \text{ as } \varepsilon
    \to 0.
  \end{align*}
  Finally, as $\E_{\xi} \Big[ \frac{\| T_{\tilde{f}} u - T_{\tilde{f}} v \|}{\|u - v \|} 
  \Big]$ is 
  finite, we can apply the dominated convergence theorem to obtain that
  \begin{align*}
    \E_{\xi} \Big[ \frac{\| T_f u - T_f v \|^2}{\|u - v \|^2} \Big]
    \leq \E_{\xi} \Big[ \frac{\| T_{\tilde{f}} u - T_{\tilde{f}} v \|}{\|u - v \|} \Big]
    \leq \Big(\E_{\xi}\Big[ \frac{\| T_{\tilde{f}} u - T_{\tilde{f}} v \|^2}{\|u - v \|^2}
    \Big]\Big)^{\frac{1}{2}}.
  \end{align*}
\end{proof}

After having established a connection between the contraction properties of $T_{f,\xi}$ and 
$T_{\tilde{f},\xi}$, the next step is to provide a concrete result for the contraction factor of 
$T_{\tilde{f},\xi}$. Applying Lemma~\ref{lem:welldef_resolvents}, 
we can express this resolvent in terms of $M_{\xi}$, which is easier to handle due to its
linearity.

\begin{lemma}\label{lem:contraction_factor}
  Let Assumption~\ref{ass:f} be satisfied and let $\tilde{f}(\cdot,\xi)$ be given as in 
  \eqref{eq:defTildeF}. Then for $u, v \in H$ and $\alpha > 0$,
  \begin{equation*}
    \E_{\xi}\big[ \| T_{\alpha\tilde{f}, \xi} u - T_{\alpha\tilde{f}, \xi} v \|^2\big]
    <  \E_{\xi}\big[ \| (I + \alpha M_{\xi})^{-1} \|_{\L(H)}^2 \big] \|u- v\|^2
  \end{equation*}
  is fulfilled.
  Furthermore, it follows that
  \begin{align*}
    \E_{\xi}\big[ \| (I + \alpha M_{\xi})^{-1} \|_{\L(H)}^2 \big] <  
    1 - 2\mu\alpha + 3\nu^2 \alpha^2.
  \end{align*}
\end{lemma}

\begin{proof}
  Due to the explicit representation of $T_{\alpha\tilde{f}, \xi}$ stated in 
  Lemma~\ref{lem:welldef_resolvents}, we find that
  \begin{equation*}
    T_{\alpha\tilde{f}, \xi} u - T_{\alpha\tilde{f}, \xi} v = (I + \alpha M_{\xi})^{-1}(u-v)
  \end{equation*}
  for $u,v \in H$. As $u-v$ does not depend on $\Omega$, it follows that
  \begin{align*}
    \E_{\xi}\big[ \| (I + \alpha M_{\xi})^{-1} (u - v) \|^2\big]
    \leq \E_{\xi}\big[ \| (I + \alpha M_{\xi})^{-1} \|_{\L(H)}^2 \big] \| u - v \|^2.
  \end{align*}
  Thus, we have reduced the problem to a question about ``how contractive'' the 
  resolvent of $M_{\xi}$ is in expectation.
  We note that for any $u \in H$, we have 
  \begin{align*}
    \inner{(I + \alpha M_{\xi})u}{u}
    \geq (1 + \mu_{\xi} \alpha) \|u\|^2.
  \end{align*}
  Due to Lemma~\ref{lem:positive_operator_inequalities} it follows 
  that
  \begin{equation*}
    \|(I + \alpha M_{\xi})^{-1}\|_{\L(H)}^2 \le (1 + \mu_{\xi} \alpha)^{-2} .
  \end{equation*}
  The right-hand-side bound is a $C^2(0,\infty)$-function with respect to $\alpha$ (in 
  fact, it is even in $C^{\infty}(0,\infty)$). By a second-order expansion in a Taylor series 
  we can therefore conclude that 
  \begin{equation*}
    \|(I + \alpha M_{\xi})^{-1}\|_{\L(H)}^2 
    \le 1 - 2 \mu_{\xi} \alpha + 3 \mu_{\xi}^2 \alpha^2.
  \end{equation*}
  Combining these results, we obtain
  \begin{align*}
    \E_{\xi} \big[\| (I + \alpha M_{\xi})^{-1} \|_{\L(H)}^2\big]
    \leq \E_{\xi} \big[ 1 - 2 \mu_{\xi} \alpha + 3 \mu_{\xi}^2 \alpha^2 
    \big]
    =  1 - 2\mu\alpha + 3\nu^2 \alpha^2.
  \end{align*}
\end{proof}

Finally, the proof of the main theorem relies on iterating the step-wise bounds arising from 
the contraction properties of the resolvents which we just established. This leads to 
certain products of the contraction factors. The following algebraic inequalities show that 
these are bounded in the desired way.

\begin{lemma} \label{lem:algebraic_inequalities}
  Let $C_1, C_2>0$, $p>0$ and $r \ge 0$ satisfy $C_1p > r$ and $4C_2 \ge C_1^2$. Then the following inequalities 
  are satisfied:
\begin{enumerate}[label=(\roman*)]
  \item $
    \prod_{j= 1}^k \big(1-\frac{C_1}{j} + \frac{C_2}{j^2}\big)^{p} \le 
    \exptextB{\frac{C_2 p \pi^2}{6}} (k+1)^{-C_1p} $,
  
  \item $
    \sum_{j=1}^k \frac{1}{ j^{1+r}} \prod_{i = j+1}^k \big(1-\frac{C_1}{i} +  
    \frac{C_2}{i^2}\big)^{p} 
  \le 2^{C_1p} \exptextB{\frac{C_2 p \pi^2}{6}} \frac{1}{C_1 p-r} (k+1)^{-r}.$

\end{enumerate}

\end{lemma}

\begin{proof}
  The proof relies on the trivial inequality $1 + u \le \exp{u}$ for $u \geq - 1$ and 
  the following two basic inequalities involving (generalized) harmonic numbers
  \begin{align*}
    \ln{(k+1)} - \ln{(m)} \le \sum_{i=m}^k \frac{1}{i} \quad 
    \text{and} \quad
    \sum_{i=1}^k{i^{C-1}} \le \frac{1}{C} (k+1)^C .
  \end{align*}
  The first one follows quickly by treating the sum as a lower Riemann 
  sum approximating the integral $\int_{m}^{k+1} {u^{-1} \diff{u}}$.
  The second one can be proved analogously by approximating the integral 
  $\int_0^{k+1} {u^{C-1}\diff{u}}$
  with an upper ($C<1$) or lower ($C>1$) Riemann sum.

  The condition $4C_2 \ge C_1^2$ implies that all the factors in the product \emph{(i)} are positive. We therefore have that
  $0\leq 1-\frac{C_1}{j} + \frac{C_2}{j^2} \le \exptext{-\frac{C_1}{j}} 
  \exptext{\frac{C_2}{j^2}}$. Thus, it follows that
  \begin{align*}
    \prod_{j= 1}^k \Big(1-\frac{C_1}{j} + \frac{C_2}{j^2}\Big)^{p}
    &\le \exptextB{-C_1 p \sum_{j=1}^k{\frac{1}{j}}} 
    \exptextB{ C_2 p \sum_{j=1}^k{\frac{1}{j^2}}} \\
    &\le \exptextB{-C_1 p \ln{(k+1)}  } \exptextB{\frac{C_2 p\pi^2}{6}},
  \end{align*}
  from which the first claim follows directly. For the second claim, we similarly 
  have
  \begin{align*}
    &\sum_{j=1}^k \frac{1}{j^{1+r}} \prod_{i = j+1}^k \Big(1-\frac{C_1}{i} + 
    \frac{C_2}{i^2}\Big)^{p}
    \le \exptextB{\frac{C_2 p \pi^2}{6}}\sum_{j=1}^k \frac{1}{j^{1+r}} 
    \exptextbigg{-C_1p \sum_{i=j+1}^k{\frac{1}{i}} } ,
  \end{align*}
  where the latter sum can be bounded by
  \begin{align*}
    \sum_{j=1}^k \frac{1}{j^{1+r}} \exptextB{-C_1p \sum_{i=j+1}^k{\frac{1}{i}} } 
    &\le \sum_{j=1}^k \frac{1}{j^{1+r}} \exptextB{-C_1p\ln\Big( 
    \frac{k+1}{j+1}  \Big) } \\
    &\le \sum_{j=1}^k{ \frac{1}{j^{1+r}} 
      \Big(\frac{k+1}{j+1}\Big)^{-C_1p}} \\
    &= (k+1)^{-C_1p} \sum_{j=1}^k{ j^{C_1 p-r-1} \cdot 
    \Big(\frac{j+1}{j}\Big)^{C_1p}} \\
    &\le \frac{2^{C_1p}}{C_1 p-r} (k+1)^{-r}.
  \end{align*}
  The final inequality is where we needed $C_1p > r$, in order to have something 
  better than $j^{-1}$ in the sum.
  \end{proof}

\section{Proof of main theorem}\label{section:main_theorem}

We are now in a position to prove Theorem~\ref{theorem:main_theorem}.

\begin{proof}[Proof of Theorem~\ref{theorem:main_theorem}]
Given the sequence of mutually independent random variables $\xi^k$, we abbreviate the 
random functions $f_k = f(\cdot, \xi^k)$ and $T_{k} = T_{\alpha_k f, \xi_k}$, $k \in \N$. 
Then the scheme can be written as
$w^{k+1} = T_k w^{k}$.  If $T_k w^* = w^*$, we would 
essentially 
only have to invoke Lemma~\ref{lem:same_contraction_factors} and 
Lemma~\ref{lem:contraction_factor} to finish the proof. But due to the 
stochasticity, this does not hold, so we need to be more careful.

We begin by adding and subtracting the term $T_k w^*$ and find that
\begin{align*}
  \|w^{k+1} - w^* \|^2 
  &= \|T_k w^{k} - T_k w^* \|^2
  + 2\inner{T_k w^{k} - T_k w^*}{T_k w^* - w^*} 
  + \|T_k w^* - w^* \|^2.
\end{align*}
By Lemma~\ref{lem:same_contraction_factors} and 
Lemma~\ref{lem:contraction_factor} the expectation $\E_{\xi^k}$ of the 
first term on the right-hand side is bounded by $(1 - 2\mu \alpha_k + 
3\nu^2\alpha_k^2)^{\nicefrac{1}{2}}\|w^{k} - 
w^* \|^2$ while by 
Lemma~\ref{lem:basic_inequalities} the last term is bounded in expectation by 
$\alpha_k^2 \sigma^2$. The second term is the problematic one. We add and subtract 
both $w^{k}$ and $w^*$ in order to find terms that we can control:
\begin{align*}
  &\inner{T_k w^{k} - T_k w^*}{T_k w^* - 
  w^*} \\
  &= \innerb{(T_k - I) w^{k} - (T_k - 
  I)w^*}{(T_k - I)w^*} + \innerb{ w^{k} - w^*}{(T_k - I)w^*} \\
  &\equalscolon I_1 + I_2.
\end{align*}

  In order to bound $I_1$ and $I_2$, we first need to apply the a priori bound from 
  Lemma~\ref{lem:apriori}. This will also enable us to utilize the local Lipschitz 
  condition. First, we notice that due to Lemma~\ref{lem:basic_inequalities}, we find that
  \begin{align*}
    \big(\E_{\xi^k} \big[\|T_kw^*\|^j\big] \big)^{\frac{1}{j}}
    \leq \|w^*\|  + \big(\E_{\xi^k} \big[ \|\nabla f_k(w^*)\|_{H^*}^j\big] 
    \big)^{\frac{1}{j}} 
    \leq \|w^*\|  + \sigma 
  \end{align*}
  is bounded for $j \leq 2^m$. As $T_k$ is a contraction, we also 
  obtain
  \begin{align*}
    \big(\E_k \big[\|T_kw^{k}\|^j\big] \big)^{\frac{1}{j}}
    &\leq \big(\E_k \big[\|T_kw^{k} - T_kw^*\|^j\big] 
    \big)^{\frac{1}{j}} + \big(\E_{\xi^k} \big[\|T_kw^*\|^j\big] 
    \big)^{\frac{1}{j}}\\
    &\leq \big(\E_{k} \big[\|w^{k} - w^*\|^j\big] 
    \big)^{\frac{1}{j}} +  \|w^*\|  + \sigma .
  \end{align*}
  Thus, there exists a random variable $R_1$ such that 
  \begin{align*}
    \max\Big( \max_{j \in \{1,\dots,k\}} \|T_j w^{j}\|, \max_{j \in \{1,\dots,k\}} \|T_j w^*\|\Big) 
    \le R_1,
  \end{align*}
  and $\E_k [R_1^j]$ is bounded for $j \leq 2^m$. For $I_1$, we then obtain that
 \begin{align*}
   I_1 &\le \innerb{(T_k - I) w^{k} - (T_k - I) w^* 
   }{(T_k - I)w^*} \\
   &\le \|\alpha_k \nabla f_k(T_kw^{k}) -  \alpha_k\nabla 
   f_k(T_kw^*)\|_{H^*} \|\alpha_k \nabla f_k(w^*) \|_{H^*}\\
   &\le \alpha_k^2 \Lr{\xi^k}{R_1} \| T_kw^{k} - T_kw^*\| 
   \|\nabla f_k(w^*) \|_{H^*}\\
   &\le \alpha_k^2 \Lr{\xi^k}{R_1} \| w^{k} - w^*\| \|\nabla f_k(w^*) \|_{H^*}, 
 \end{align*}
 where we used the fact that $T_k$ is non-expansive in the last step.
 Taking the expectation, we then have by H\"older's inequality that
 \begin{align*}
   \E_k [I_1 ]
   &\leq \alpha_k^2 \E_k \big[\Lr{\xi^k}{R_1} \| w^{k} - w^*\| \|\nabla f_k(w^*)   
   \|_{H^*}\big]\\
   &\leq \alpha_k^2 \tilde{L}_1
   \big(\E_{k-1} \big[\| w^{k} - w^*\|^{2^m} \big]\big)^{2^{-m}} 
   \big(\E_{\xi_k} \big[\|\nabla f_k(w^*)   
   \|_{H^*}^{2^m} \big]\big)^{2^{-m}},
 \end{align*}
 where
 \begin{align*}
  \tilde{L}_1 =
  \begin{cases}
    \big(\E_k \big[P(R_1)^{\frac{2^{m-1}}{2^{m-1} -1}} 
    \big]\big)^{\frac{2^{m-1} -1}{2^{m-1}}}, \quad &m > 1,\\
    \sup |P(R_1)|, &m = 1.
  \end{cases}
 \end{align*}
 As $P$ is a polynomial of at most order $2^m -2$, the exponent for $P$ is 
 bounded by $\big(\frac{2^{m-1}}{2^{m-1} -1}\big) \big(2^m 
 -2\big) = 2^{m}$. Hence $\tilde{L}_1$ is bounded, and in view of Lemma~\ref{lem:apriori} we get that
 \begin{equation*}
   \E_k[I_1] \le D_1 \alpha_k^2,
 \end{equation*}
where $D_1 \geq 0$ is a constant depending only on $\|w^*\|$, $\|w_1 - w^*\|$, $\sigma$ 
and $\eta$.
For $I_2$, we add and subtract $\alpha_k \iota \nabla f_k w^*$ to get
\begin{align*}
  I_2&= \innerb{ w^{k} - w^*}{(T_k - I)w^*} \\
  &= \innerb{ w^{k} - w^*}{(T_k - I)w^* + \alpha_k \iota \nabla f_k 
  (w^*) }  - \innerb{ w^{k} - w^*}{\alpha_k \iota \nabla f_k (w^*)} .
\end{align*}
Since $w^{k} - w^*$ is independent of $\alpha_k \nabla f_k (w^*)$, it follows 
that
\begin{align*}
  \E_{\xi_k} [\innerb{ w^{k} - w^*}{\alpha_k \iota \nabla f_k (w^*)} ]
  = \innerb{ w^{k} - w^*}{ \E_{\xi_k} [\alpha_k \iota \nabla f_k (w^*)]} = 0.
\end{align*}
Using the Cauchy-Schwarz inequality and Lemma~\ref{lem:basic_inequalities}, we find that
\begin{align*}
  \E_k [I_2] &\le \E_k \big[ \| w^{k} - w^*\| \| \iota^{-1} (T_k -  
  I)w^* + \alpha_k \nabla f_k (w^*)\|_{H^*} \big]\\
  &\le \E_k\big[ \Lr{\xi^k}{R_2} \alpha_k^2 \|w^{k} - w^*\| \| \nabla f_k (w^*)\|_{H^*} 
  \big] \\
  &\le \alpha_k^2 \tilde{L}_2
  \big(\E_{k-1}\big[\|w^{k} - w^*\|^{2^m}\big] \big)^{2^{-m}}
  \big(\E_{\xi_k} \big[\|\nabla f_k(w^*)   
  \|_{H^*}^{2^m} \big]\big)^{2^{-m}},
\end{align*}
where $R_2 = \max (\|w^*\|,\|\nabla f_k (w^*)\|_{H^*})$ and
\begin{align*}
  \tilde{L}_2 =
  \begin{cases}
  \big(\E_k \big[P(R_2)^{\frac{2^{m-1}}{2^{m-1} -1}} 
  \big]\big)^{\frac{2^{m-1} -1}{2^{m-1}}}, \quad &m > 1,\\
  \sup |P(R_2)|, &m = 1.
  \end{cases}
\end{align*}
Just as for $I_1$, we therefore get by Lemma~\ref{lem:apriori} that 
 \begin{equation*}
   \E_k[I_2] \le D_2 \alpha_k^2,
 \end{equation*}
where $D_2 \geq 0$ is a constant depending only on $\|w^*\|$, $\|w_1 - w^*\|$, $\sigma$ 
and $\eta$.

Summarising, we now have
\begin{align*}
  \E_k\big[ \|w^{k+1} - w^* \|^2 \big] &\le \tilde{C}_k \E_{k-1}\big[\|w^{k} - w^* 
  \|^2\big] +  \alpha_k^2 D 
\end{align*}
with $\tilde{C}_k = \big(1 - 2\mu \alpha_k + 3\nu^2\alpha_k^2 
\big)^{\nicefrac{1}{2}}$ 
and $D = \sigma^2 + D_1 + D_2$.
Recursively applying the above bound yields
\begin{align*}
  \E_k\big[ \|w^{k+1} - w^* \|^2 \big] \le \prod_{j=1}^k{ \tilde{C}_j \|w_1 - w^* 
  \|^2} 
  + D \sum_{j=1}^k{ \alpha_j^2 \prod_{i=j+1}^k{ \tilde{C}_i}  }.
\end{align*}
Applying Lemma~\ref{lem:algebraic_inequalities} (i) and (ii) with 
$p=\nicefrac{1}{2}$, $r=1$, $C_1 = 2\mu\eta$ and $C_2 = 3\nu^2\eta^2$ then shows that 
\begin{align*}
  \prod_{j=1}^k{ \tilde{C}_j} \le \exptextB{\frac{\nu^2\eta^2 
      \pi^2}{4}} (k+1)^{-\mu\eta}
\end{align*}
and
\begin{align*} 
  \sum_{j=1}^k{ \alpha_j^2 \prod_{i=j+1}^k{ \tilde{C}_i}}  
    \le \eta^2 2^{\mu \eta} \exptextB{\frac{\nu^2\eta^2 \pi^2}{4}} \frac{1}{ \mu \eta - 1} 
    (k+1)^{-1}.
\end{align*}
Thus, we finally arrive at
\begin{equation*}
  \E_k\big[ \|w^{k+1} - w^*\|^2 \big] \le \frac{C}{k+1},
\end{equation*}
where $C$ depends on $\|w^*\|$, $\|w_1 - w^*\|$,  $\mu$, $\sigma$ and $\eta$.
\end{proof}

\begin{remark}
  The above proof is complicated mainly due to the stochasticity and due to the 
  lack of strong convexity. We consider briefly the simpler, deterministic,  
  full-batch, case with
  \begin{equation*}
    w^{k+1} = w^k - \alpha_k \nabla F(w^{k+1}),
  \end{equation*}
where $F$ is strongly convex with convexity constant $\mu$. Then it can 
easily be shown that 
\begin{align*}
    \inner{\nabla F(v) - \nabla F(w)}{v - w} \geq \mu \| v - w\|^2.
  \end{align*}
This means that
\begin{align*}
    \|\big( I + \alpha \nabla F \big)^{-1}(v) - \big( I + \alpha \nabla F
    \big)^{-1}(w)\|
    \leq (1 + \alpha \mu )^{-1} \|v - w\|,
  \end{align*}
i.e.\ the resolvent is a strict contraction. Since $\nabla F(w^*) = 0$, it follows that
$\big( I + \alpha \nabla F \big)^{-1} w^* = w^*$ so a simple iterative argument 
shows that
\begin{align*}
    \|w^{k+1} - w^*\|^2
    \leq \prod_{j=1}^k \big(1 + \alpha_j \mu \big)^{-1} \| w_1 -
    w^*\|^2.
  \end{align*}
Using $(1 + \alpha \mu)^{-1} \le 1 - \mu \alpha + \mu^2\alpha^2$, 
choosing $\alpha_k = \eta/k$ and applying 
Lemma~\ref{lem:algebraic_inequalities} then shows that
\begin{equation*}
  \|w^{k+1} - w^*\|^2 \le C (k+1)^{-1}
\end{equation*}
for appropriately chosen $\eta$. In particular, these arguments do not require the 
Lipschitz continuity of $\nabla F$, which is needed in the stochastic case to 
handle the terms arising due to $\nabla f(w^*, \xi) \neq 0$.

\end{remark}

\section{Numerical experiments}\label{section:numerical_experiments}
In order to illustrate our results, we set up a numerical experiment along the lines 
given in the introduction. In the following, let $H = L^2(0,1)$ be the Lebesgue 
space of square integrable functions equipped with the usual inner product and 
norm. Further, let $x_j^i \in H$ for $i = 1$, $j = 1, \dots, \lfloor\frac{n}{2}\rfloor$ and $i = 
2$, $j = \lfloor\frac{n}{2}\rfloor + 1, \ldots, n$ be elements from two different classes within 
the space $H$. In particular, we 
choose each $x_j^1$ to be a polynomial of degree $4$ and each $x_j^2$ to 
be a trigonometric function with bounded frequency for $j = 1,\dots,n$. The 
polynomial coefficients 
and the frequencies were randomly chosen. 

We want to classify these functions as either polynomial or trigonometric. To do 
this, we set up an affine (SVM-like) classifier by choosing the loss function 
$\ell(h,y)= \ln(1 + \exp{-hy})$ and the prediction function $h( [w,\overline{w}], 
x) = \inner{w}{x} + \overline{w}$ with $[w,\overline{w}] \in L^2(0,1) \times \R$. 
Without $\overline{w}$, this would be linear, but by
including $\overline{w}$ we can 
allow for a constant bias term and thereby make it affine. 
We also add a regularization term 
$\frac{\lambda}{2} \|w\|^2$ (not including the bias), such that the minimization 
objective is
\begin{equation*} 
  F([w,\overline{w}], \xi) = \frac{1}{n} { \sum_{j = 1}^{n}{\ell(h( [w,\overline{w}], x_j), 
  y_j)} + \frac{\lambda}{2} \|w\|^2 },
\end{equation*}
where $[x_j, y_j] = [x^1_j, -1]$ if $j \le \lfloor\frac{n}{2}\rfloor$ and $[x_j, y_j] = [x^2_j, 1]$ 
if $j > \lfloor\frac{n}{2}\rfloor$, similar to Equation~\eqref{eq:ML_functional}. In one step of 
SPI, we use the function
\begin{equation*} 
  f([w,\overline{w}], \xi) = \ell(h( [w,\overline{w}], x_{\xi}), y_{\xi}) + 
  \frac{\lambda}{2} 
  \|w\|^2 ,
\end{equation*}
with a random variable $\xi \colon \Omega \to \{1,\dots,n\}$.
Since we cannot do computations 
directly in the infinite-dimensional space, we discretize all the functions using 
$N$ equidistant points in $[0,1]$, omitting the endpoints. For each $N$, this 
gives us an optimization problem 
on $\mathbb{R}^N$, which approximates the problem on $H$. 

For the implementation, we make use of the following computational idea, which 
makes SPI essentially as fast as SGD. Differentiating the chosen $\ell$ and $h$ 
shows that the scheme is given by the iteration
\begin{align*}
  [w,\overline{w}]^{k+1} 
  = [w,\overline{w}]^{k} + c_k [x_k,1] - \lambda \alpha_k [w,0]^{k+1},
\end{align*}
where $c_k =  \frac{\alpha_k  y_k }{ 1 + \exptext{ \inner{w^{k+1}}{x_k} y_k + 
\overline{w}^{k+1} y_k}}$. This is equivalent to
\begin{align*}
  w^{k+1} 
  = \frac{1}{1 + \alpha_k \lambda} \big(w^{k} +  c_k  x_k\big)
  \quad 
  \text{and}
  \quad
  \overline{w}^{k+1} 
  = \overline{w}^{k} +  c_k.
\end{align*}
Inserting the expression for $[w,\overline{w}]^{k+1}$ in the definition of $c_k$, 
we obtain that
\begin{equation*}
  c_k =   \frac{\alpha_k y_k}{ 1 + \exptextB{ \frac{1}{1 + \alpha_k 
        \lambda} \inner{w^{k} + c_k  x_k}{x_k} y_k + (\overline{w}^{k} +  c_k) y_k} }.
\end{equation*}
We thus only need to solve one scalar-valued equation. This is at most twice as 
expensive as SGD, since the equation solving is essentially free and the only 
additional costly term is  $\inner{x_k}{x_k}$ (the term $\inner{w^k}{x_k}$ of 
course has to be computed also in SGD). By storing the scalar result, the extra 
cost will be essentially zero if the same sample is revisited. We note that 
extending this approach to larger batch-sizes is straightforward. If the batch size 
is $B$, then one has to solve a $B$-dimensional equation.

Using this idea, we implemented the method in Python and tested it on a series of 
different discretizations. We took $n = 1000$, i.e.\ $500$ functions of each type, 
$M = 10000$ time steps and discretization parameters $N = 100 \cdot 2^i$ for $i = 1, 
\ldots, 11$ to approximate the infinite dimensional space $L^2(0,1)$. We used 
$\lambda = 10^{-3}$ and the initial step size $\eta = \frac{2}{\lambda}$, since in 
this case it can be shown that $\mu \ge \lambda$. There is no closed-form 
expression for the exact minimum $w^*$, so instead we ran SPI with $10 M$ time 
steps and used the resulting reference solution as an approximation to 
$w^*$. Further, we approximated the expectation $\E_k$ by running the 
experiment $10$ times and averaging the resulting errors. This may seem like a 
small number of paths but using more (or less) 
such paths does not seem to influence the results much, indicating that the 
convergence is likely actually almost surely rather than only in expectation. In 
order to 
compensate for the vectors becoming longer as $N$ increases, we measure the 
errors in the RMS-norm $\| \cdot \|_N = \| \cdot \|_{\R^N} / \sqrt{N+1}$. As $N 
\to \infty$, this tends to the $L^2$ norm.

Figure~\ref{fig:errors_SPI} shows the resulting approximated errors 
$\E_{k-1}[\|w^{k} - w^*\|_N^2]$. As expected, we observe convergence 
proportional to $\nicefrac{1}{k}$ for all $N$. The error constants do vary to a 
certain extent, 
but they are reasonably similar. As the problem approaches the 
infinite-dimensional case, they vary less. In order to decrease the computational 
requirements, we only compute statistics at every $100$ time steps, this is why 
the plot starts at $k = 100$.

\begin{figure}[ht]
  \begin{center}
    \includegraphics[width=\columnwidth]{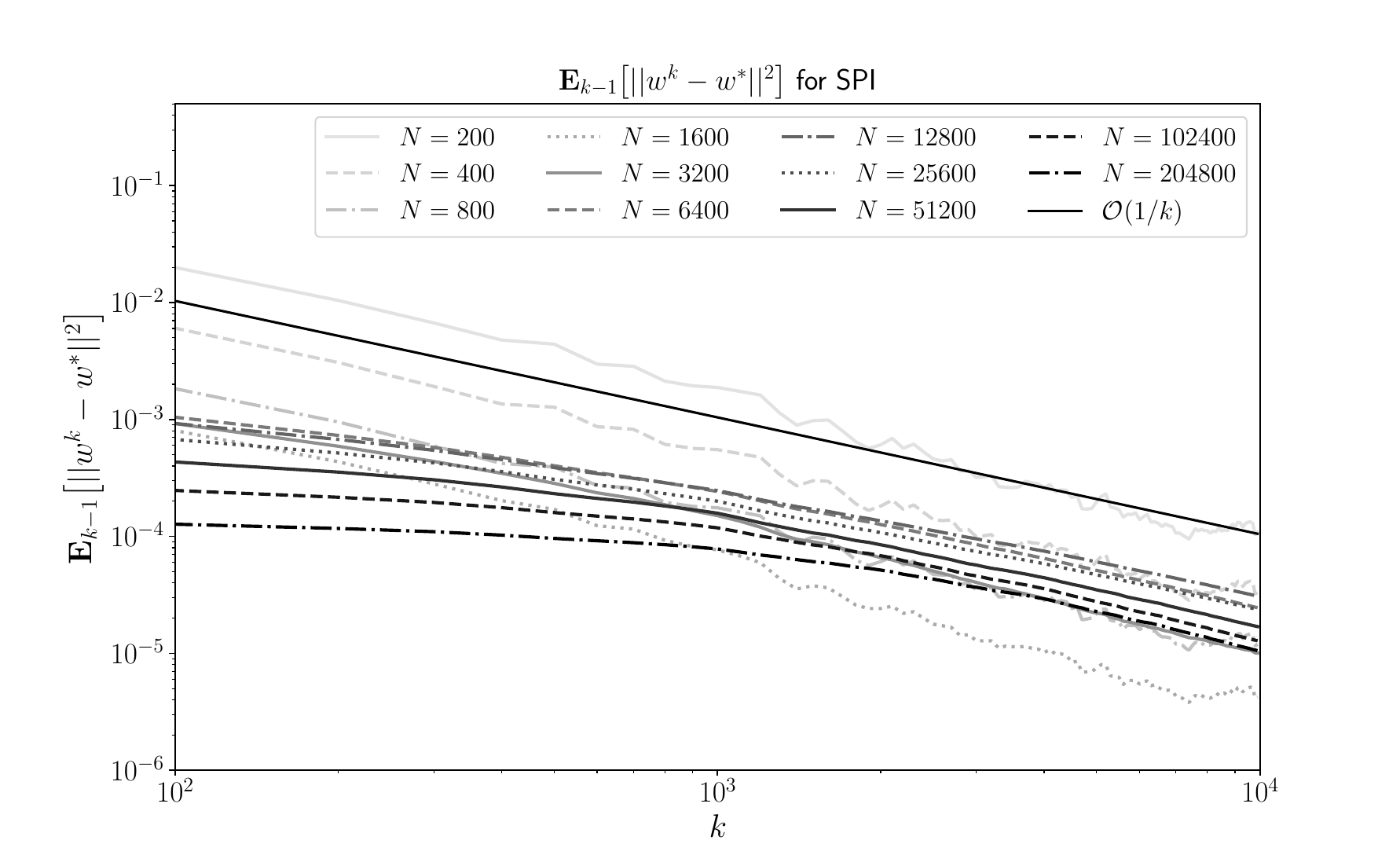}
  \end{center}
  \caption{Approximated errors $\E_{k-1}[\|w^{k} - w^*\|_N^2]$ for the SPI 
  method, measured in RMS-norm, for discretizations with varying number of grid 
  points $N$. Statistics were only computed at every $100$ time steps, this is why 
  the plot starts at $k = 100$. The $1/k$-convergence is clearly seen by 
  comparing to the uppermost solid black reference line.}
\label{fig:errors_SPI}
\end{figure}

In contrast, redoing the same experiment but with the explicit SGD method 
instead results in Figure~\ref{fig:errors_SGD}. We note that except for $N = 200$ 
and $N=400$, the method does not converge at all, likely because as $N$ grows 
the problem also becomes more stiff. Even when it does converge, the errors are 
much larger than in Figure~\ref{fig:errors_SPI}. Many more steps would be 
necessary to reach the same accuracy as SPI. Since our implementations are 
certainly not optimal in any sense, we do not show a comparison of computational 
times here. They are, however, very similar, meaning that SPI is more efficient than 
SGD for this problem.

\begin{figure}[ht]
  \begin{center}
    \includegraphics[width=\columnwidth]{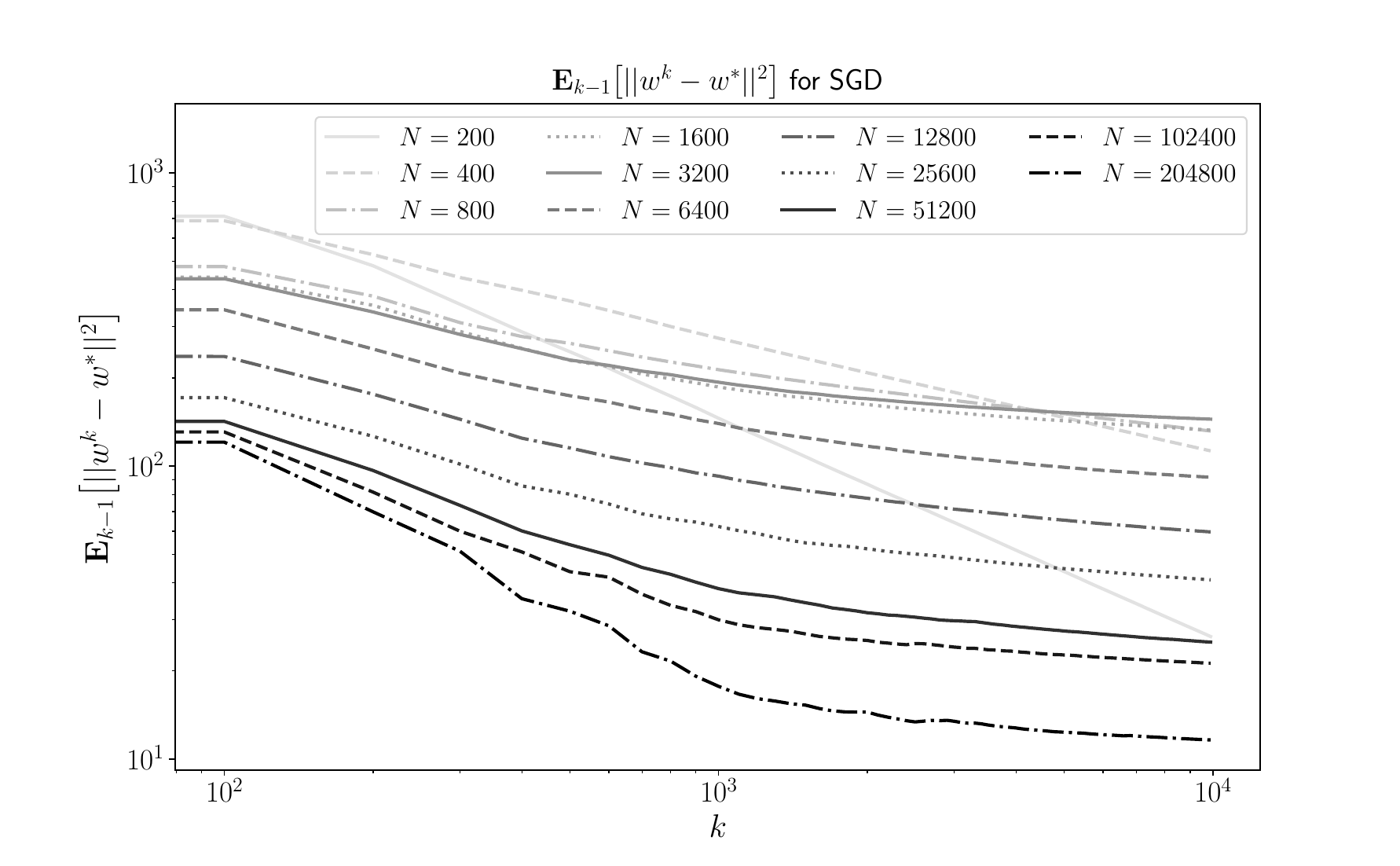}
  \end{center}
  \caption{Approximated errors $\E_{k-1}[\|w^{k} - w^*\|_N^2]$ for the 
  SGD method, measured in RMS-norm, for discretizations with varying number of 
  grid points $N$. Statistics were only computed at every $100$ time steps, this is 
  why the plot starts at $k = 100$. Except for $N = 200$ and $N=400$, the 
  method does not converge at all. Even when it does, the errors are much larger 
  than in Figure~\ref{fig:errors_SPI}.}
\label{fig:errors_SGD}
\end{figure}

\section{Conclusions}\label{section:conclusions}
We have rigorously proved convergence with an optimal rate for the stochastic 
proximal iteration method in a general Hilbert space. 
This improves the analysis situation in two ways. Firstly, by providing an 
extension of similar results in a finite-dimensional setting to the 
infinite-dimensional case, as well as extending these to less bounded operators. 
Secondly, by improving on similar infinite-dimensional results that only achieve 
convergence, without any error bounds. The latter improvement comes at the cost 
of stronger assumptions on the cost functional. Global Lipschitz continuity of the 
gradient is, admittedly, a rather strong assumption. However, as we have 
demonstrated, this can be replaced by local Lipschitz continuity where the 
maximal growth of the Lipschitz constant is determined by higher moments of the 
gradient applied to the minimum. This is a weaker condition. Finally, we have seen 
that the theoretical results are applicable also in practice, as demonstrated by the 
numerical results  in the previous section.


\end{document}